\documentclass{gtpart}     
\usepackage{pinlabel}  
\usepackage{graphicx}  

\usepackage{amsmath}
\usepackage{amssymb}
\usepackage{amsthm}

\title{Functions on surfaces and constructions of manifolds in dimensions three, four and five}

\author{David T Gay}
\givenname{David T}
\surname{Gay}
\address{Euclid Lab\\ 160 Milledge Terrace\\ Athens, GA 30606\\\newline Department of Mathematics\\ University
  of Georgia\\ Athens, GA 30602}
\email{d.gay@euclidlab.org}

\volumenumber{}
\issuenumber{}
\publicationyear{}
\papernumber{}
\startpage{}
\endpage{}
\doi{}
\MR{}
\Zbl{}
\received{}
\revised{}
\accepted{}
\published{}
\publishedonline{}
\proposed{}
\seconded{}
\corresponding{}
\editor{}
\version{}

\newtheorem{theorem}{Theorem}
\newtheorem{lemma}[theorem]{Lemma}

\theoremstyle{definition}
\newtheorem{definition}[theorem]{Definition}

\def\Z{\mathbb Z}

\def\R{\mathbb R}
\def\C{\mathbb C}

\def\e{\epsilon}

\def\cP{\mathcal{P}}

\newcommand{\id}{\mathop{\rm id}\nolimits}

\newcommand{\Crit}{\mathop{\rm Crit}\nolimits}

\begin{document}

\begin{abstract}    
 We offer a new proof that two closed oriented $4$--manifolds are cobordant if their signatures agree, in the spirit of Lickorish's proof~\cite{Lickorish} that all closed oriented $3$--manifolds bound $4$--manifolds. Where Lickorish uses Heegaard splittings we use trisections. In fact we begin with a subtle recasting of Lickorish's argument: Instead of factoring the gluing map for a Heegaard splitting as a product of Dehn twists, we encode each handlebody in a Heegaard splitting in terms of a Morse function on the surface and build the $4$--manifold from a generic homotopy between the two functions. This extends up a dimension by encoding a trisection of a closed $4$--manifold as a triangle (circle) of functions and constructing an associated $5$--manifold from an extension to a $2$--simplex (disk) of functions. This borrows ideas from Hatcher and Thurston's proof~\cite{HatcherThurston} that the mapping class group of a surface is finitely presented.
\end{abstract}

\maketitle


The central theme of this paper is that manifolds of dimension $n+3$ can be described by generic $n$--parameter families of functions on $2$--manifolds (at least for $n \leq 2$). After some preparatory discussion and definitions, this is formulated as our main theorem, Theorem~\ref{T:Universal}. The following well known theorems are then proved as corollaries:

\begin{theorem}[Rohlin~\cite{RohlinO3}, Thom~\cite{Thom}] \label{T:3D}
 For every closed, oriented $3$--manifold $M$ there is a compact, oriented $4$--manifold $X$ with $\partial X = M$.
\end{theorem}

\begin{theorem}[Pontrjagin, Rohlin~\cite{RohlinO4}] \label{T:4D}
 For every closed, oriented $4$--manifold $X$ there is a compact, oriented $5$--manifold $Z$ with $\partial Z = X \amalg (\amalg^p \C P^2) \amalg (\amalg^q \overline{\C P^2})$, for some $p, q \geq 0$.
\end{theorem}

Theorem~\ref{T:3D} is compactly stated as $\Omega_3=0$, i.e. the oriented $3$--dimensional cobordism group is trivial. Theorem~\ref{T:4D} is usually stated by saying that $\Omega_4=\Z$, i.e. that the oriented $4$--dimensional cobordism group is $\Z$, with generator $\C P^2$; this follows from our statement and the fact that the signature is a $4$--dimensional cobordism invariant~\cite{Thom}.

Our proof of Theorem~\ref{T:3D} is very similar in spirit to Lickorish's proof~\cite{Lickorish}, but our proof of Theorem~\ref{T:4D} is quite different from other proofs, and depends on the existence of trisections of $4$--manifolds~\cite{GayKirbyTrisections}, in the same way that Lickorish's proof depends on the existence of Heegaard splittings of $3$--manifolds. In a nutshell, the idea is to use generic $B^n$-- or $S^n$--parameterized families of smooth functions on closed, connected, oriented surfaces to construct $(n+3)$--manifolds. We work only in the cases where $n \leq 2$, but leave open the interesting question of to what extent these ideas extend to higher dimensions.

The author is aware of two other papers which relate the construction and existence of cobordisms with ``higher Morse theory'', but in subtly different ways: In~\cite{CostantinoThurston}, $4$--manifolds bounded by given $3$--manifolds are constructed using information coming from generic (stable) smooth maps from $3$--manifolds to $\mathbb{R}^2$. In~\cite{SaekiCobordism}, it is shown how to reconstruct a $4$--manifold up to $5$--dimensional cobordism from data associated with stable smooth maps from $4$--manifolds to $3$--manifolds; from this one sees that $\mathbb{C}P^2$ generates $\Omega_4$.

Throughout this paper all manifolds are smooth and oriented. Unless explicitly stated otherwise, all surfaces ($2$--manifolds) are closed, connected and oriented, but manifolds in other dimensions may have nonempty boundary. (In fact, our manifolds with boundary may have corners as well, and a ``diffeomorphism'' between two manifolds with boundary is allowed to bend at the corners. It should be clear from context when this is happening and that this does not raise any significant issues.)

We will work with smooth parameterized families of real-valued functions on a surface $\Sigma$; when the parameter space is $A$, we will denote such a family as a function $F: A \to C^\infty(\Sigma)$, always with the understanding that this really means that the associated function $A \times \Sigma \to \R$ is smooth. Our parameter spaces will always be spheres or balls. Given $F: A \to C^\infty(\Sigma)$ we will use the notation $\partial F : \partial A \to C^\infty(\Sigma)$ to refer to $F|_{\partial A}$.

Recall~\cite{Sharko}, when the dimension of $A$ is $0$, $1$ or $2$, that such a parameterized family of functions $F: A \to C^\infty(\Sigma)$ is {\em stable} if for each $p \in A$ and $q$ in $\Sigma$ we can find coordinates near $p$, $q$ and $F(p,q) \in \R$ with respect to which $F: A \times \Sigma \to \R$ has one of the following models, listed according to the dimension of $A$ (where $(x,y)$ are the coordinates on $\Sigma$):
\begin{enumerate}
 \setcounter{enumi}{-1}
 \item When $\dim(A)=0$ we have two options:
 \begin{description}
  \item[Regular point] $(x,y) \mapsto x$
  \item[Morse critical point] $(x,y) \mapsto \pm x^2 \pm y^2$
 \end{description}
 \item When $\dim(A)=1$ we have three options (here $t$ is the coordinate on $A$):
 \begin{description}
  \item[Path of regular points] $(t,x,y) \mapsto x$
  \item[Path of Morse critical points] $(t,x,y) \mapsto \pm x^2 \pm y^2$
  \item[Birth/death event] $(t,x,y) \mapsto x^3 + t x \pm y^2$, with $t=0$ not on $\partial A$.
 \end{description}
 \item When $\dim(A)=2$ we have four options (here $s$ and $t$ are the coordinates on $A$):
 \begin{description}
  \item[Disk of regular points] $(s,t,x,y) \mapsto x$
  \item[Disk of Morse critical points] $(s,t,x,y) \mapsto \pm x^2 \pm y^2$
  \item[Path of birth/death events] $(s,t,x,y) \mapsto x^3 + t x \pm y^2$, with $\{t=0\}$ transverse to $\partial A$.
  \item[Swallowtail] $(s,t,x,y) \mapsto x^4 + s x^2 + t x \pm y^2$, with $s=t=0$ not on $\partial A$.
 \end{description}
\end{enumerate}

Furthermore, given such a family $F : A \to C^\infty(\Sigma)$ we can also think of $F$ as a function $F: A \times \Sigma \to A \times \R$, being $\id_A$ on the first factor, and consider the locus of critical points $\Crit(F) \subset A \times \Sigma$; this is a smooth codimension $2$ submanifold mapped by $F$ to $A \times \R$. We then impose the following transversality conditions, again listed according to the dimension of $A$:
\begin{enumerate}
 \setcounter{enumi}{-1}
 \item If $\dim(A)=0$ then, for each $p \in A$, we require that $F$ is injective on $\Crit(F)$, i.e. all Morse critical points have distinct critical values
 \item If $\dim(A)=1$, then $\Crit(F) = \Crit_0(F) \cup \Crit_1(F)$, where $\Crit_0(F)$ is the locus of critical points that arise in dimension $0$, i.e. Morse critical points, and $\Crit_1(F)$ is the locus of birth/death events. Thus $\Crit_0(F)$ is $1$--dimensional and $\Crit_1(F)$ is $0$--dimensional. Here we require that:
 \begin{enumerate}
  \item $F$ is injective on $\Crit_1(F)$,
  \item the image of $\Crit_1(F)$ is disjoint from the image of $\Crit_0(F)$ and from $\partial A \times \R$, and
  \item $\Crit_0(F)$ is immersed with at worst transverse double point self-intersections, which should occur away from $\partial A \times \R$. (We call this event, which occurs at isolated parameter values, a {\em critical value height switch}, or simply a {\em height switch}.)
 \end{enumerate}
 \item If $\dim(A) = 2$ then $\Crit(F) = \Crit_0(F) \cup \Crit_1(F) \cup \Crit_2(F)$, where $\Crit_0(F)$ is the locus of critical points that arise in dimension $0$, $\Crit_1(F)$ is the locus of birth/death events, and $\Crit_2(F)$ is the locus of swallowtails. Now $\Crit_0(F)$ is $2$--dimensional, $\Crit_1(F)$ is $1$--dimensional and $\Crit_2(F)$ is $0$--dimensional. We now require that:
 \begin{enumerate}
  \item $F$ is injective on $\Crit_2(F)$,
  \item the image of $\Crit_2(F)$ is disjoint from the images of $\Crit_1(F)$ and $\Crit_0(F)$ and from $\partial A \times \R$,
  \item the image of $\Crit_1(F)$ is embedded and transverse to the image of $\Crit_0(F)$ and to $\partial A \times \R$, and
  \item $\Crit_0(F)$ is immersed with at worst arcs of transverse double points (the arcs being transverse to $\partial A \times \R$) and isolated transverse triple points (away from $\partial A \times \R$). The arcs of transverse double points are thus $1$--parameter families of critical value height switches and the isolated triple points correspond to three critical values switching heights in a $2$--parameter fashion, to be discussed in more detail later.
 \end{enumerate}
\end{enumerate}

We will abuse the term ``generic'' in reference to such families.
\begin{definition} \label{D:generic}
 If $A \in \{B^0,S^0,B^1,S^1,B^2\}$, an $A$-parameterized family of functions on a surface $\Sigma$, namely $F: A \to C^\infty(\Sigma)$, is {\em generic} if it is stable and satisfies the appropriate transversality condition on $\Crit(F)$ described above. For purely formal completeness, if $A = S^{-1}$ we declare there to be a unique $S^{-1}$--parametrized family, and declare it to be {\em generic}.
\end{definition}
 
All we really need in this paper are the appropriate (and standard) relative existence statements: Every generic $G : \partial A \to C^\infty(\Sigma)$ extends to a generic $F : A \to C^\infty(\Sigma)$ with $\partial F = G$.

\begin{definition} \label{D:MSPair}
 A {\em manifold-surface pair} is a pair $(Y,\Sigma)$, where $Y$ is a manifold of some dimension $\geq 2$ and $\Sigma \subset Y$ is a surface submanifold of $Y$. If $\partial Y \neq \emptyset$, we require that $\Sigma \subset \partial Y$. We define the {\em boundary of a manifold-surface pair} $\partial (Y, \Sigma)$ to be either $(\emptyset,\emptyset)$ if $\partial Y = \emptyset$ or to be $(\partial Y, \Sigma)$ if $\partial Y \neq \emptyset)$. An embedding (resp. diffeomorphism) $\phi: Y \to Y'$ is an {\em embedding (resp. diffeomorphism) of pairs} $(Y,\Sigma) \to (Y,\Sigma')$ if $\phi(\Sigma) = \Sigma'$, and is an embedding (resp. diffeomorphism) {\em rel $\Sigma$} if $\Sigma = \Sigma'$ and $\phi|_\Sigma$ is the identity map. The {\em dimension} of a pair $(Y,\Sigma)$ is the dimension of $Y$.
\end{definition}

We now present the main idea of this paper in the form of a single theorem, and then we will prove Theorems~\ref{T:3D} and ~\ref{T:4D} as immediate corollaries.

\begin{theorem} \label{T:Universal}
 For each $A \in \{S^{-1},B^0,S^0,B^1,S^1,B^2\}$, with $\dim(A) = n \in \{-1,0,1,2\}$, there exists a function $\cP_A = \cP$ (we generally drop the subscript) from the set of generic $A$-parameterized functions on surfaces to the set of $(n+3)$--dimensional manifold-surface pairs, satisfying the following properties:
 \begin{enumerate}
  \item If $A = S^{-1} = \emptyset$ and $F: A \to C^\infty(\Sigma)$ is the unique $A$-parameterized family of functions on a fixed surface $\Sigma$, then $\cP(F) = (\Sigma,\Sigma)$. \label{I:Sm1}
  \item For $A \in \{S^{-1},B^0, S^0, B^1, S^1\}$, $\partial \cP(F) = \cP (\partial F)$. (When $\partial A = \emptyset$, $\cP(\partial F) = \emptyset$.) \label{I:bound}
  \item For $A = B^2$, for every $\Sigma$ and every generic $F: A \to C^\infty(\Sigma)$ there are non-negative integers $p$ and $q$ such that $\partial \cP(F) = \cP(\partial F) \amalg (\amalg^p (\C P^2,\emptyset)) \amalg (\amalg^q (\overline{\C P^2},\emptyset))$. \label{I:5bound}
  \item For every closed connected $(n+3)$--manifold $M$, where $n \in \{-1,0,1\}$, there is a surface $\Sigma \subset M$ and a generic $F: S^n \to C^\infty(\Sigma)$ such that $(M,\Sigma)$ is diffeomorphic rel. $\Sigma$ to $\cP(F)$. \label{I:all234d}
 \end{enumerate}
\end{theorem}

In fact
$\cP_A$ can be constructed for $A=S^2$ using our techniques but we have no use for this case in this paper, and cannot state an existence result like item~\ref{I:all234d} above, so we leave this for others to investigate more thoroughly, along with higher dimensional parameter spaces.

\begin{proof}[Proof of Theorem~\ref{T:3D}]
 Given a closed connected $3$--manifold $M$, choose a surface $\Sigma \subset M$ and a generic $F: S^0 \to C^\infty(\Sigma)$ such that $(M,\Sigma)$ is diffeomorphic rel. $\Sigma$ to $\cP(F)$. Extend $F$ to a generic $G: B^1 \to C^\infty(\Sigma)$ with $G|_{\partial B^1} = F$, and let $(X,\Sigma) = \cP(G)$. Then $X$ is a $4$--manifold and $\partial X = M$.
\end{proof}

\begin{proof}[Proof of Theorem~\ref{T:4D}]
 Given a closed connected $4$--manifold $X$, choose a surface $\Sigma \subset X$ and a generic $F: S^1 \to C^\infty(\Sigma)$ such that $(X,\Sigma)$ is diffeomorphic rel. $\Sigma$ to $\cP(F)$. Extend $F$ to a generic $G: B^2 \to C^\infty(\Sigma)$ with $G|_{\partial B^2} = F$, and let $(Z,\Sigma) = \cP(G)$. Then $Z$ is a $5$--manifold and $\partial Z = X \amalg (\amalg^p \C P^2) \amalg (\amalg^q \overline{\C P^2})$.
\end{proof}

\begin{proof}[Proof of Theorem~\ref{T:Universal}]
 This proof occupies the remainder of the paper but we give the outline here, to be filled in by means of explicit constructions and lemmas later. There is nothing to prove for item~\ref{I:Sm1}, as this constitutes the definition of the function $\cP_{S^{-1}}$. Items~\ref{I:bound}
 and~\ref{I:5bound} will all be proved as part of the construction of each $\cP_A$ for $A \in \{B^0,S^0,B^1,S^1,B^2\}$. Item~\ref{I:all234d} is trivial in the case $n=-1$, is explained for $n=0$ in the construction of $\cP_{S^0}$ as a simple consequence of the existence of Heegaard splittings of $3$--manifolds, and will be proved for $n=1$ via a sequence of lemmas after the constructions of $\cP_{B^1}$ and $\cP_{S^1}$, ultimately as a consequence of the existence of trisections of $4$--manifolds.
\end{proof}


\begin{definition} \label{D:RegComponent}
 A {\em regular level component} of a smooth real-valued function $f$ on a surface $\Sigma$ is a connected component $C$ of a level set $f^{-1}(z)$ of $f$ such that $C$ contains no critical points of $f$. Note that $z \in \R$ may be a critical value, but none of its critical preimages may be in the component $C$. A {\em critical level component} of $f$ is a connected component of a level set containing at least one critical point.
\end{definition}

Of course, if $\Sigma$ is closed then all regular level components are smooth circles. Critical level components may be arbitrarily complicated depending on the nature of the singularities involved.

\begin{definition}
 Given a surface $\Sigma$, a {\em handlebody filling} of $\Sigma$ is a pair $(H,\Sigma)$, where $H$ is a $3$--dimensional handlebody with $\Sigma = \partial H$.
\end{definition}

\begin{lemma} \label{L:FcnToHBody}
Given a surface $\Sigma$ and a Morse function $f: \Sigma \to \R$ which is injective on critical points, there is a handlebody filling $(H,\Sigma)$ of $\Sigma$ with the property that every regular level component of $f$ bounds a disk in $H$. Any two such handlebody fillings of $\Sigma$ are diffeomorphic rel. $\Sigma$, and any handlebody filling of $\Sigma$ arises in this way from some such Morse function on $\Sigma$.
\end{lemma}

\begin{proof}
 The only ambiguity in the definition of $H$ comes from critical level components of $f$. Suppose $p \in \Sigma$ is a critical point of $f$, with $f(p)=a$, and let $\e>0$ be such that $[a-\e,a+\e]$ contains no critical values other than $a$. Let $\Gamma$ be the component of $f^{-1}(a)$ containing $p$ and let $N$ be the component of $f^{-1}([a-\e,a+\e])$ containing $p$. All boundary components of $N$ are regular level components of $f$ and thus bound disks in $H_f$; if we can show that $N$ is a subsurface of $\Sigma$ of genus $0$, then all simple closed curves on $N$ will automatically bound in $H_f$ and thus there will in fact be no ambiguity in the definition of $H_f$. Because $f$ is injective on critical points, $p$ is the only critical point in $N$, and $\Gamma$ is the only critical level set in $N$. If $p$ is a critical point of index $0$ or $2$, $N$ is a disk and thus has genus $0$. If $p$ has index $1$ then $\Gamma$ is a figure eight graph. Because $\Sigma$ is oriented and because $\Gamma$ separates $N$ into two halves $N \cap f^{-1}([0,a+\e])$ and $N \cap f^{-1}([a-\e,0])$, we see that $N$ is a pair of pants, with either two boundary components in $f^{-1}(a-\e)$ and one in $f^{-1}(a+\e)$ or vice versa. Thus $N$ again has genus $0$.
 
 To see that every handlebody filling of $\Sigma$ arises this way, note that for each $g$ there exists a standard genus $g$ handlebody $H_g$ embedded in $\R^3$ such that the standard height function is a Morse function on $\partial H_g$ with regular level components bounding disks in $H_g$. Given a particular handlebody filling $(H,\Sigma)$, choose a diffeomorphism of $H$ with $H_g$ and pull back the height function on $\partial H_g$ to $\Sigma$.
\end{proof}

Note that this lemma can fail if $f$ is not injective on critical points, even if $f$ is Morse, because components of critical level sets containing several critical points may have higher genus neighborhoods.

\begin{proof}[Construction of $\cP_{B^0}$:]
 A generic $B^0$--parameterized family $F$ is nothing more than a single Morse function $f = F(0)$ on a surface $\Sigma$ which is injective on critical points. Define $\cP(F) = \cP_{B^0}(F)$ to be the handlebody filling $(H,\Sigma)$ uniquely characterized by $f$ as in the preceding lemma. 
\end{proof}

\begin{proof}[Construction of $\cP_{S^0}$:]
 Given a family of functions $F: S^0 = \{-1,1\} \to C^\infty(\Sigma)$, let $\cP_{S^0}(F)$ be the closed $3$--manifold obtained by gluing the two handlebodies $\cP(F_{-1})$ and $\cP(F_1)$ together along their common boundary $\Sigma$. More precisely, with orientations in mind:
 \[ \cP(F) = -\cP(F|_{-1}) \cup_\Sigma \cP(F|_1) \]
 Here it is important that the boundaries of these handlebodies are both {\em equal} to $\Sigma$, not just diffeomorphic to $\Sigma$, so that we can really glue by the identity map. The fact (item~\ref{I:all234d} in Theorem~\ref{T:Universal}) that every closed $3$--manifold arises this way is simply the fact that every $3$--manifold has a Heegaard splitting together with the part of Lemma~\ref{L:FcnToHBody} asserting that all handlebody fillings come from Morse functions. 

\end{proof}

Next we consider generic $S^1$-- and $B^1$--parametrized families. Note that, in a generic family $F: A \to C^\infty(\Sigma)$, with $A=S^1$ or $A=B^1$, $A$ can be subdivided into subintervals $A=A_1 \cup \ldots \cup A_n$ such that each $F|_{A_i}$ has the following behavior: (1) $F|_{\partial A_i}$ is generic, i.e. a pair of Morse functions injective on critical points. (2) In the interior of each $A_i$ there is at most one value of $t$ at which $F(t)$ is not Morse and injective on critical points; if there is such a value, label it $c_i$. (3) For each $c_i$, as $t$ varies from $c_i - \epsilon$ to $c_i+\epsilon$ for some small $\epsilon > 0$, either $F(t)$ experiences a birth or death of a single cancelling pair of critical points (so that $F(c_i)$ is not Morse) or exactly two critical points switch heights (so that $F(c_i)$ is Morse but not injective on critical points). 

\begin{definition} \label{D:ElemSub}
Each such subinterval $A_i$ in a subdivision of $A$ as above will be called an {\em elementary subinterval} and we say that $F|_{A_i}$ is an {\em elementary $1$--parameter family}. 
\end{definition}

In preparation for the upcoming discussion of $\cP_{S^0}$ and $\cP_{B^1}$ applied to elementary $1$--parameter families, we name some specific manifold-surface pairs. Fix integers $g \geq k \geq 0$.
\begin{enumerate}
 \item The $3$--dimensional pair $(\#^k S^1 \times S^2, \#^k S^1 \times S^1)$ refers to the standard genus $k$ Heegaard splitting of $\#^k S^1 \times S^2$. This can be explicitly described by identifying $\#^k S^1 \times S^2$ with $\partial ([-1,1] \times (\natural^k S^1 \times B^2))$, in which case $\#^k S^1 \times S^1$ is really $\{0\} \times (\#^k S^1 \times S^1$).
 \item The corresponding $4$--dimensional pair is $(\natural^k S^1 \times B^3,\#^k S^1 \times S^1)$, with boundary equal to $(\#^k S^1 \times S^2, \#^k S^1 \times S^1)$.
 \item The $3$--dimensional pair $(\#^k S^1 \times S^2, \#^g S^1 \times S^1)$ refers to the result of stabilizing (in the sense of Heegaard splittings) $(\#^k S^1 \times S^2, \#^k S^1 \times S^1)$ exactly $g-k$ times.
 \item The corresponding $4$--dimensional pair is $(\natural^k S^1 \times B^3,\#^g S^1 \times S^1)$, with boundary equal to $(\#^k S^1 \times S^2, \#^g S^1 \times S^1)$.
\end{enumerate}

\begin{lemma}
 Suppose the surface $\Sigma$ is a surface of genus $g$, and that $F: A \to C^\infty(\Sigma)$ is an elementary $1$--parameter family. Then $\cP(\partial F)$ is diffeomorphic to either $(\#^g S^1 \times S^2, \#^g S^1 \times S^1)$ or $(\#^{g-1} S^1 \times S^2, \#^g S^1 \times S^1)$.
\end{lemma}

\begin{proof}
 If we can show that the handlebodies $\cP(F|_{\{-1\}})$ and $\cP(F|_{\{1\}})$ are diffeomorphic rel. $\Sigma$, then we know that $\cP(\partial F) \cong (\#^g S^1 \times S^2, \#^g S^1 \times S^1)$. We will show this in all but one case.

 When $F(t)$ is Morse and injective on critical points for all $t$ in $A$, the regular level components of $F(t)$ move by an isotopy of $\Sigma$, so that the handlebodies $\cP(F|_{\{-1\}})$ and $\cP(F|_{\{1\}})$ are diffeomorphic rel. $\Sigma$.

 Otherwise, let $t_* \in A$ be the parameter value at which the birth/death or height switch occurs, and let $f_* = F(t_*): \Sigma \to \R$.

 In the case of a birth/death event, we claim again that $\cP(F|_{\{-1\}})$ and $\cP(F|_{\{1\}})$ are diffeomorphic rel. $\Sigma$. To see this, let $\Gamma$ be the critical level component of $f_*$ containing the birth/death point $p$, assume that $f_*(p)=0$, and let $N$ be the component of $f_*^{-1}[-\e,\e]$ containing $p$. (Choose $\e$ small enough so that $N$ contains no other critical points, of course.) From the local model of a birth/death, we see that $\Gamma$ is a circle (smooth except at $p$) and $N$ is a smooth cylinder with $\partial N$ equal to two regular level components of $f_*$. We may assume that the regular level components of $F(t)$ do not change (except by isotopy) outside $P$ as $t$ ranges from $t < t_*$ to $t>t_*$ and thus, since the genus of $N$ is $0$, we do not need to know anything about regular level components in $N$ to know the handlebody filling.
 
 When the singular event is a critical value crossing, we have further cases to consider. Let the two critical points whose critical values cross be labelled $p$ and $q$. If $p$ and $q$ occur in different critical level components of their common level set for $f_*$, then changing their relative heights can only change regular level components by isotopies and thus does not change the handlebodies, and once again $\cP(F|_{\{-1\}})$ and $\cP(F|_{\{1\}})$ are diffeomorphic rel. $\Sigma$. If either $p$ or $q$ has index $0$ or $2$, this certainly happens. If $p$ and $q$ both have index $1$ and occur in the same component with, say, $f_*(p)=f_*(q)=0$, let $N$ be the component of $f_*^{-1}[-\e,\e]$ containing $p$ and $q$, where $[-\e,\e]$ contains no other critical values, and let $\Gamma$ be the component of $f_*^{-1}(0)$ containing $p$ and $q$. Then $N$ is a regular neighborhood of the graph $\Gamma$, which has two $4$--valent vertices $p$ and $q$. To construct all possible pairs $(N,\Gamma)$ like this, we can start with one edge between $p$ and $q$ and construct the various possibilities from there. Figure~\ref{F:SwitchOptions} enumerates the possibilities, bearing in mind the constraints that $N$ should be orientable and that $\Gamma$ should separate $N$ into a positive half $N \cap f_*^{-1}[0,\e]$ and a negative half $N \cap f_*^{-1}[-\e,0]$. We end up with four options, three of which have genus $0$ and one of which is genus $1$.
 \begin{figure}[h!]
 \centering
 \includegraphics{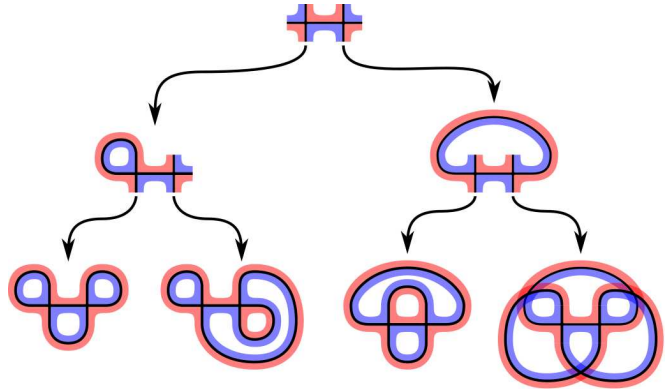}
 \caption{ Enumerating possible regular neighborhoods of a critical level set containing two index $1$ critical points.} \label{F:SwitchOptions}
 \end{figure} 
 In the genus $0$ case, we again claim, for the same reason as in the birth/death case, that $\cP(F|_{\{-1\}})$ and $\cP(F|_{\{1\}})$ are diffeomorphic rel. $\Sigma$. This is because we can assume that $F(-1)$ and $F(1)$ agree outside $N$ (up to isotopy), and thus the circles on $\Sigma \setminus N$ which bound disks in $\cP(F|_{\{-1\}})$ also bound disks in $\cP(F|_{\{1\}})$ and vice versa. Since $N$ has genus $0$, specifying the curves that bound in $\Sigma \setminus N$ unambiguously determines a handlebody filling of $\Sigma$ up to diffeomorphism rel. $\Sigma$.

 In the center in Figure~\ref{F:Genus1Switch} we have redrawn the genus $1$ case, with the two critical points labelled $p$ and $q$. The genus $1$ surface is obtained from the hexagon by identifying opposite edges.
 \begin{figure}[h!]
	\labellist
		\pinlabel $p$ [tl] at 156 57
		\pinlabel $q$ [tr] at 185 57
	\endlabellist
 \centering
 \includegraphics{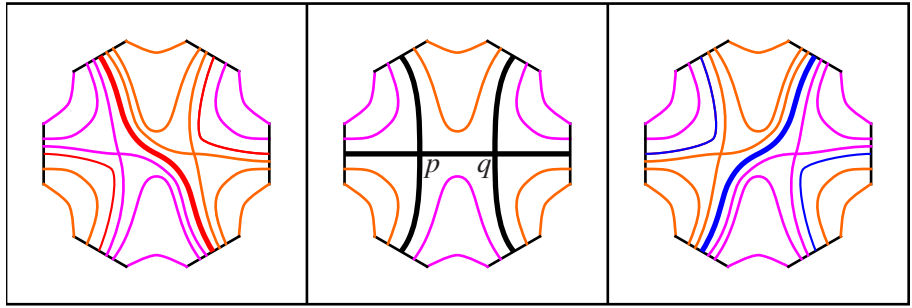}
 \caption{ A movie of the genus $1$ case for a critical value crossing.} \label{F:Genus1Switch}
 \end{figure}
 Note that $N$ has two boundary components, $f_*^{-1}(\pm \e)$. To the left and right we have drawn various regular and critical level sets as we perturb $f_*$ to a nearby function $f$ so as to either make $f(p)$ positive and $f(q)$ negative or vice versa. Positive level sets are orange while negative level sets are pink; we highlight the $0$ level set in red on the left, black in the middle and blue on the right. Thus these are three frames in a movie of $F(t)$, as $t$ ranges from $t_* - \e$ to $t_*+\e$. Let $H_{\pm 1} = \cP(F|_{\{\pm 1\}})$, the ``before'' and ``after'' handlebody fillings of $\Sigma$. Note that the thicker red and blue level components, if drawn on the same surface, intersect at a point; these are colored red and blue to suggest $\alpha$ and $\beta$ curves in a Heegaard diagram. The upshot is that $\Sigma$ contains an embedded genus $1$ surface $N$ with two boundary components such that:
 \begin{enumerate}
  \item Every curve in $\Sigma \setminus N$ which bounds a disk in $H_{-1}$ also bounds a disk in $H_1$, and vice versa.
  \item Both boundary components of $N$ bound disks in both $H_{-1}$ and $H_1$.
  \item There are essential simple closed curves $C_{-1}$ and $C_1$ in $N$ which intersect transversely at one point such that $C_{-1}$ bounds a disk in $H_{-1}$ and $C_1$ bounds a disk in $H_1$.
 \end{enumerate}
Note that since both boundary components of $N$ bound disks in $H_{\pm 1}$, any band sum in $N$ of these two boundary components bounds a disk in $H_{\pm 1}$, and thus we may replace $N$ with a genus $1$ surface with one boundary component, and still keep items (1) and (3) from the list above. From this we see that $\cP(\partial F)$ comes with a Heegaard splitting $(-H_{-1}) \cup_\Sigma H_1$ which is the result of stabilizing a genus $g-1$ splitting in which every curve that bounds a disk in one half of the splitting bounds in the other half as well. This implies that $\cP(\partial F) \cong (\#^{g-1} S^1 \times S^2, \#^g S^1 \times S^1)$.
\end{proof}

\begin{definition} \label{D:Types01}
 An elementary sub-interval $A_i$, or elementary $1$--parameter family $F|_{A_i}$, is of {\em type 0} if, in the above theorem, we have $\cP(F|_{\partial A_i}) \cong (\#^g S^1 \times S^2, \#^g S^1 \times S^1)$, and is of {\em type 1} if $\cP(F|_{\partial A_i}) \cong (\#^{g-1} S^1 \times S^2, \#^g S^1 \times S^1)$.
\end{definition}

\begin{proof}[Construction of $\cP_{B^1}$:]
 Given a generic $F: B^1 \to C^\infty(\Sigma)$, subdivide $B^1=[-1,1]$ into elementary subintervals $A_i = [t_{i-1},t_i]$, with $t_0=-1 < t_1 < \ldots < t_n=1$. Choose $\epsilon > 0$ such that, on each interval $A_i$, the parameter value $c_i$ at which $F$ fails to be Morse and injective on critical points is in the interval $(t_{i-1}+\epsilon,t_i-\epsilon)$. We begin the construction of $\cP(F)$ with $X_* = [-\epsilon,0] \times [-1,1] \times \Sigma$. For each $t_i$, let $(H_i,\Sigma) = \cP(F|_{\{t_i\}})$ be the handlebody filling of $\Sigma$ determined by the Morse function $F(t_i) : \Sigma \to \R$. Let $X_0 = [t_0,t_0+\epsilon] \times H_0$, let $X_n = [t_n-\epsilon,t_n] \times H_n$, and for each $i=1,\ldots,n-1$ let $X_i = [t_i-\epsilon,t_i+\epsilon] \times H_i$. Let $X_{**}$ be the result of gluing each $X_i$ to $X_*$ by identifying $\{t\} \times \Sigma \subset H_i$ with $\{-\epsilon\} \times \{t\} \times \Sigma \subset X_*$.
 
 Note that now $\partial X_{**}$ has $n+1$ boundary components, each obtained by gluing two handlebody fillings of $\Sigma$ together. One of them is $\cP(F|_{\partial B^1})$, and the other $n$ are $\cP(F|_{\partial A_i})$ for $i = 1, \ldots, n$. Since each $A_i$ is an elementary subinterval we can fill in each $\cP(F|_{\partial A_i})$ with either $\natural^g S^1 \times B^3$ (in the type $0$ case) or $\natural^{g-1} S^1 \times B^3$ (in the type $1$ case); the result is our $4$--manifold $X$, with $\cP(F) = (X,\Sigma)$, where $\Sigma$ is $\{(0,0)\} \times \Sigma \subset X_* \subset X$. Note that these final fillings are unique due to Laudenbach and Poenaru's result~\cite{LaudenbachPoenaru} that every self-diffeomorphism of $\#^n S^1 \times S^2$ extends to $\natural^n S^1 \times B^3$.
 
 A finer subdivision of $B^1$ into elementary subintervals only results in more type $0$ subintervals, and since $\natural^g S^1 \times B^3 \cong [0,1] \times (\natural^g S^1 \times B^2)$, it is easy to see that creating more type $0$ subintervals does not change the resulting $4$--manifold-surface pair in our construction.

\end{proof}

\begin{proof}[Construction of $\cP_{S^1}$]
 Given a generic $F: S^1 \to C^\infty(\Sigma)$, simply split $S^1$ into two $B^1$'s on which $F$ remains generic (i.e. split at parameter values where $F$ is Morse and injective on critical points), as $S^1 = B^1_- \cup B^1_+$, and glue $-\cP(F|_{B^1_-})$ to $\cP(F|_{B^1_+})$ along their common boundary $\cP(F|_{S^0})$, where the $S^0$ is $\partial B^1_+ = - \partial B^1_-$. (Note that we are gluing manifold-surface pairs, so the gluings should match the $\Sigma$ on one side with the $\Sigma$ on the other side, via the identity map. Since $\Sigma$ splits each $3$--dimensional boundary into two handlebodies, and since a self-diffeomorphism of a handlebody is determined up to isotopy by its restriction to the boundary, the gluings here are well defined.)
 Item~\ref{I:bound} 
 of Theorem~\ref{T:Universal} is clear, but the important thing that remains in this dimension is to prove item~\ref{I:all234d}, namely that all closed $4$--manifolds can be produced this way. This happens in the following lemmas.
\end{proof}

\begin{lemma} \label{L:S1XB3sFromPB1}
 Let $X = \natural^k S^1 \times B^3$ and let $\Sigma$ be a genus $g \geq k$ Heegaard surface in $\#^k S^1 \times S^2 = \partial X$. Then there exists a generic $B^1$--parameterized family of functions $F$ on $\Sigma$ such that $(X,\Sigma)$ is diffeomorphic rel. $\Sigma$ to $\cP(F)$.
\end{lemma}

\begin{proof}
 Consider the standard genus $g$ Heegaard diagram $(\Sigma, \alpha, \beta)$ for $\#^k S^1 \times S^2 = \partial X$ shown in Figure~\ref{F:Stdgk}, and let $f_0 : \Sigma \to \R$ be the (Morse) height function as indicated in the figure, so that all the $\alpha$ curves are regular level components of $f_0$. For $i = k+1, \ldots, g$, the dual pair $(\alpha_i,\beta_i)$ lies in a twice-punctured torus $T_i$ with both components of $\partial T_i$ being regular level components of $f_0$, and with no other $\alpha_j$ or $\beta_j$ intersecting $T_i$. Let $f_t$, $t \in [0,1]$, be the generic family such that, for $i = k+1, \ldots, g$, for  $t \in [(i-k-1)/(g-k),(i-k)/(g-k)]$, $f_t$ does not vary with $t$ outside $T_i$ and, inside $T_i$, $f_t$ undergoes a critical value crossing that changes from having $\alpha_i$ as a regular level component to having $\beta_i$ as a regular level component. In other words, the associated family $F: [0,1] \to C^\infty(\Sigma)$ is elementary of type $1$ on each subinterval $[(i-k-1)/(g-k),(i-k)/(g-k)]$, switching $\alpha_i$ to $\beta_i$ as regular level components, exactly as in Figure~\ref{F:Genus1Switch}.
 \begin{figure}[h!]
	\labellist
	\endlabellist
 \centering
 \includegraphics{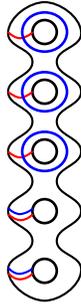}
 \caption{Standard genus $g$ Heegaard diagram for $\#^k S^1 \times S^2$, with the vertical height function so that the $\alpha$ curves (red) are all regular level components. (Here $g=5$ and $k=2$.)} \label{F:Stdgk} 
 \end{figure}
 
 We now claim that, for this $F$, we have that $\cP(F)$ is diffeomorphic rel. $\Sigma$ to $(X,\Sigma)$. To see this, think of $\cP(F)$ as a cobordism (with sides) from the $3$--dimensional handlebody $\cP(F|_{0})$ to the handlebody $\cP(F|_{1})$ built by stacking the elementary cobordisms corresponding to each elementary subinterval $[t_{i-1},t_i]$, with $t_i = (i-k)/(g-k)$. Each of these elementary cobordisms is simply a $4$--dimensional $2$--handle attached to $[t_{i-1},t_{i-1}+\epsilon] \times H$, where $H$ is a genus $g$ handlebody, and because the surgery corresponding to this $2$--handle attachment results in switching from $\alpha_i$ bounding a disk to $\beta_i$ bounding a disk, we see that this $2$--handle simply cancels the $4$--dimensional $1$--handle in $[t_{i-1},t_{i-1}+\epsilon] \times H$ dual to $\alpha_i$. Thus, the full cobordism contains $(g-k)$ $2$--handles each cancelling one of the $1$--handles in $[0,\epsilon] \times (\natural^g S^1 \times B^2)$, and this is precisely $X$.
\end{proof}

\begin{lemma} \label{L:TrivPath}
 Consider generic $G, G': B^0 \to C^\infty(\Sigma)$ such that the handlebodies $\cP(G)$ and $\cP(G')$ are diffeomorphic rel. $\Sigma$. Then there exists a generic family $F: B^1 \to C^\infty(\Sigma)$, with $F|_{-1} = G$ and $F|_{1} = G'$ such that any subdivision of $B^1$ into elementary subintervals for $F$ involves only subintervals of type $0$. In particular, $\cP(F)$ is a product, i.e. if $\cP(G) = (H,\Sigma)$ then $\cP(F)$ is diffeomorphic rel. $\Sigma$ to $([-1,1] \times H, 0 \times \Sigma)$.
\end{lemma}

\begin{proof}
 Let $f = G(0)$ and $f' = G(0)$, generic Morse functions on $\Sigma$ with the property that there is a single handlebody filling $H$ of $\Sigma$ such that  all the regular level components of $f$ and $f'$ bound disks in $H$. We will construct a generic $1$--parameter family $f_t$, $t \in [-1,1]$, with $f_{-1}=f$ and $f_1 = f'$, together with a subdivision of $[-1,1]$ into subintervals $[t_{i-1},t_i]$ such that, on each interval $[t_{i-1},t_i]$, $f_t$ does not vary with $t$ outside a genus $0$ subsurface of $\Sigma$ bounded by regular level components. This suffices to prove the lemma since $f_t$ changing with $t$ inside a genus $0$ subsurface cannot change the associated handlebody filling.
 
 Let $C$ be a cut system for $H$ made of level components for $f$, and let $C'$ be a similar cut system for $f'$. Let $C_0 = C, C_1, C_2, \ldots, C_n=C'$ be a sequence of cut systems for $H$, each related to the next by a single handle slide~\cite{Johannson}. Then there exist Morse functions $f_1, f_2, \ldots, f_n$ such that both $C_0$ and $C_1$ are collections of regular level components for $f_1$, both $C_1$ and $C_2$ are collections of levels for $f_2$, and so on, until $C_{n-1}$ and $C_n$ are levels for $f_n$. This reduces our problem to having to interpolate from $f=f_0$ to $f_1$, from $f_1$ to $f_2$, and so on, in each case interpolating between two functions which have a single common cut system made of level components for both functions.
 
 So now suppose $f$ and $f'$ are given and there is a single cut system $C$ such that each curve in $C$ is a level component for both $f$ and $f'$. After a generic homotopy supported in a neighborhood of each component of $C$ (i.e. only modifying functions in genus $0$ subsurfaces) we may assume $f$ and $f'$ agree on a neighborhood of $C$. Since $C$ is a cut system, the complement of this neighborhood is also genus $0$, so we may now homotope $f$ to $f'$ via a generic homotopy supported on this genus $0$ subsurface, and we are done.
\end{proof}

\begin{proof}[Proof of item~\ref{I:all234d} of Theorem~\ref{T:Universal}, when $n=1$]
 We need to prove that, for every closed $4$--manifold $X$, there is a surface $\Sigma \subset X$ and a generic $S^1$--parameterized family $F: S^1 \to C^\infty(\Sigma)$ such that $(X,\Sigma)$ is diffeomorphic rel. $\Sigma$ to $\cP(F)$.
 
 Let $X = X_1 \cup X_2 \cup X_3$ be a {\em $(g,k)$--trisection} of $X$; this is a decomposition such that each $X_i$ is diffeomorphic to $\natural^k S^1 \times B^3$, each $X_i \cap X_j$ is diffeomorphic to the $3$--dimensional handlebody $\natural^g S^1 \times B^2$, and $X_1 \cap X_2 \cap X_3$ is a genus $g$ surface. The existence of trisections on all closed $4$--manifolds is proved in~\cite{GayKirbyTrisections}. Let $\Sigma = X_1 \cap X_2 \cap X_3$. Applying Lemma~\ref{L:S1XB3sFromPB1} to each $X_i$, we get three families $F_1, F_2, F_3: B^1 \to C^\infty(\Sigma)$ such that $(X_i,\Sigma)$ is diffeomorphic rel. $\Sigma$ to $\cP(F_i)$. We cannot glue these together to form a $S^1$--parameterized family because we do not know, for instance, that $F_1|_{1} = F_2|_{-1}$ or that $F_2|_{1} = F_3|_{-1}$. We do know, however, that $\cP(F_i|_{1})$ is diffeomorphic rel. $\Sigma$ to $\cP(F_{i+1}|_{-1})$, both being diffeomorphic rel. $\Sigma$ to the handlebody $(X_i \cap X_{i+1}, \Sigma)$. Thus, by Lemma~\ref{L:TrivPath} we can interpolate between $F_i|_{1}$ and $F_{i+1}|_{-1}$ to create a $S^1$--parameterized family $F$ which is obtained by gluing the $X_i$ together with extra collars on the handlebodies between them, and the result is diffeomorphic rel. $\Sigma$ to $(X,\Sigma)$.
\end{proof}

Finally we consider $2$--parameter families; in this paper we are only concerned with the case of $B^2$--parameterized families. Note that when $F$ is a generic $B^2$--parameterized family of functions on $\Sigma$, there is a decomposition of $B^2$ into polygonal disks $B^2 = P_1 \cup \ldots \cup P_n$ which meet in pairs along intervals (edges) and in triples at points (vertices) such that:
\begin{enumerate}
 \item $F$ is generic on each polygon, along each edge and at each vertex.
 \item For each edge $e$, $F|_e$ is an elementary $1$--parameter family.
 \item For each polygon $P_i$, there is at most one point $p_i \in P_i$ at which one of the following occurs:
 \begin{enumerate}
  \item A swallowtail.
  \item A transverse triple point of self-intersection of the image of $\Crit_0(F)$ in $B^2 \times \R$, i.e. three Morse critical points with the same critical value.
  \item A transverse intersection of $\Crit_1(F)$ with $\Crit_0(F)$, i.e. a birth/death and a Morse critical point with the same critical value.
 \end{enumerate}
\end{enumerate}

\begin{definition} \label{D:ElemPoly}
Each such polygon $P_i$ in a subdivision of $A$ as above will be called an {\em elementary polygon} and we say that $F|_{A_i}$ is an {\em elementary $2$--parameter family}. 
\end{definition}
We now prepare to understand the $4$--manifolds obtained by applying $\cP_{S^1}$ to the boundary of an elementary $2$--parameter family. To do this, we first describe and name some model $4$--manifold-surface pairs and associated $5$--manifold-surface pairs. 
\begin{enumerate}
 \item Given a genus $g$, let $\Sigma_g = \#^g S^1 \times S^1$, which is the boundary of the handlebody $H_g = \natural^g S^1 \times B^2$. Let $Z_g = D^2 \times H_g$, identify $\Sigma_g$ with $\{0\} \times \Sigma_g \subset Z_g$, and let $X_g = \partial Z_g$. This gives us pairs $(X_g,\Sigma_g) = \partial (Z_g, \Sigma_g)$, with $X_g \cong \#^g S^1 \times S^3$ and $Z_g \cong \natural^g S^1 \times B^4$.
 \item Recall that the $4$--manifold-surface pair associated to a type $1$ elementary $1$--parameter family is $(\natural^{g-1} S^1 \times B^3, \#^g S^1 \times S^1)$, where the surface $\#^g S^1 \times S^1$ is the once-stabilized Heegaard surface in $\#^{g-1} S^1 \times S^2 = \partial \natural^{g-1} S^1 \times B^3$. From this we construct a $5$--manifold-surface pair $(Z_{g-1},\Sigma_g)$ where $Z_{g-1} = [-1,1] \times \natural^{g-1} S^1 \times B^3$ and $\Sigma_g = \{0\} \times \#^g S^1 \times S^1$, and we let $(X_{g-1},\Sigma_g) = \partial (Z_{g-1},\Sigma_g)$. Note that $(X_{g-1},\Sigma_g)$ can also naturally be seen as the double of $(\natural^{g-1} S^1 \times B^3, \#^g S^1 \times S^1)$, and that $X_{g-1}  \cong \#^{g-1}
 S^1 \times S^3$ and $Z_{g-1} \cong \natural^{g-1} S^1 \times B^4$.
 \item Let $T = \{[z_0:z_1:1] \mid |z_0|=|z_1|=1\}$ be the standard Lagrangian torus in $\C P^2$, giving us the $4$--manifold-surface pair $(\C P^2, T)$  and consider the $5$--manifold-surface pair $([0,1] \times \C P^2, T=\{0\} \times T)$. Note that $\partial ([0,1] \times \C P^2, \{0\} \times T) = (\C P^2, T) \amalg (\C P^2, \emptyset)$.
 \item Finally, let $X^\pm_g = X_{g-1} \#(\pm \C P^2)$, let $\Sigma^\pm_g = \Sigma_{g-1} \# T$ and let $Z^\pm_g = Z_{g-1} \natural ([0,1] \times (\pm \C P^2))$, where the boundary connected sum occurs along the boundary component $\{0\} \times (\pm \C P^2)$ in the $[0,1] \times \C P^2$ summand. This yields the pairs $(X^\pm_g,\Sigma^\pm_g)$ and $(Z^\pm_g,\Sigma^\pm_g)$, with $\partial (Z^\pm_g,\Sigma^\pm_g) = (X^\pm_g,\Sigma^\pm_g) \amalg (\pm \C P^2, \emptyset)$.
\end{enumerate}

\begin{lemma}
 If $P_i$ is an elementary polygon for a generic $F: B^2 \to C^\infty(\Sigma)$, with $\Sigma$ a surface of genus $g$, then $\cP(F|_{\partial P_i})$ is diffeomorphic to either $(X_g,\Sigma_g)$, $(X_{g-1},\Sigma_g)$ or $(X^\pm_g,\Sigma^\pm_g)$.
\end{lemma}

\begin{proof}
 To repeat, there is at most one point $p_i \in P_i$ at which one of the following occurs:
 \begin{enumerate}
  \item A swallowtail.
  \item A transverse triple point of self-intersection of the image of $\Crit_0(F)$ in $B^2 \times \R$, i.e. three Morse critical points with the same critical value. We will call this a {\em triple critical value switch}.
  \item A transverse intersection of $\Crit_1(F)$ with $\Crit_0(F)$, i.e. a birth/death and a Morse critical point with the same critical value. We will call this a {\em birth-Morse crossing}
 \end{enumerate}
 
 First we consider the case where there is no such point $p_i$. Then either $F|_{t}$ is Morse and injective on critical points for all $t \in \partial P_i$ or exactly two critical value switches occur in $F|_{\partial P_i}$. In both cases, $\cP(F|_{\partial P_i})$ is a double of $\cP(F|_{e})$ where $e$ is one edge of $P_i$, and thus the result is either $(X_g,\Sigma_g)$ or $(X_{g-1},\Sigma_g)$.
 
 When there is such a point $p_i$, first consider a swallowtail. From the local model of a swallowtail, the critical level component containing the swallowtail is a figure eight graph just like the critical level component containing a Morse critical point in the $B^0$--parameterized case. (Note that the swallowtail is the only critical point in this level set.) In particular, the critical level component has a genus $0$ neighborhood, so the handlebody filling is the same for any small perturbation of the swallowtail and thus $\cP(F|_{\partial P_i})$ is $(X_g,\Sigma_g)$.
 
 In a birth-Morse crossing, we have a Morse critical value which a birth/death event crosses in a $2$--parameter sense. At the parameter value when the Morse critical point and the birth/death point are at the same level, even if they are in the same component of the same level, that component is still a figure eight (with one non-smooth point, the birth/death), and it's neighborhood is thus genus $0$. Therefore the birth-Morse crossing can be taken to be a $2$--parameter family of functions which does not change outside a genus $0$ subsurface, so that $\cP(F|_{\partial P_i})$ is $(X_g,\Sigma_g)$.
 
 Thus the only thing interesting that can happen happens with a triple critical value switch. To analyze this, we consider the critical level component of $f_* = F|_{\{p_i\}}$ containing three index $1$ Morse critical points. (If we do not have all three critical points in the same component, we are again in one of the two easy cases $(X_g,\Sigma_g)$ or $(X_{g-1},\Sigma_g)$. If any have index $0$ or $2$, they cannot all lie in one component.) This is a $4$--valent graph $\Gamma$ with three vertices $p$, $q$ and $r$ (the three critical points) and an oriented regular neighborhood $N$ such that $\Gamma$ separates $N$ into two components $N_\pm$. Euler characteristic considerations show that $P$ can either be genus $0$ with $5$ boundary components or genus $1$ with $3$ boundary components and, again, we are only interested in the genus $1$ case. (Genus $0$ means that the handlebody filling is the same for all perturbations of the function, so that we get $(X_g,\Sigma_g)$.) The polygon $P_i$ must be a hexagon, since there are six ways to order the three critical points $p$, $q$, $r$ by height, and each edge of the hexagon represents a $1$--parameter critical value switch. Figure~\ref{F:3Perm} illustrates a neighborhood of the singular point $p_i$ at the center of $P_i$, labelled by {\em ordered partitions} of the set $\{p,q,r\}$ of critical points. The orderings are given by the Morse function $f_t$ for $t \in P_i$, and are abbreviated in the figure with the following notation: $pq|r$ means that $f_t(p)=f_t(q)<f_t(r)$, or $pqr$ means $f_t(p)=f_t(q)=f_t(r)$, or $r|q|p$ means $f_t(r)<f_t(q)<f_t(p)$. Thus the Morse function at the center of the figure has all three critical points at the same level, and along the six black rays two critical points are at the same level while one is above or below that level, and between the rays all three critical points have distinct critical values. The hexagon $P_i$ is dual to the decomposition shown; the edges of $P_i$ are perpendicular to the six rays shown. (This hexagon is otherwise known as the {\em permutahedron} on three letters, a clue for how one might approach higher dimensional cases.)
\begin{figure}[h!]
 \labellist
  \small
  \pinlabel* \fcolorbox{black}{white}{$pqr$} at 50 57
  \pinlabel* \fcolorbox{black}{white}{$p|qr$} at 50 88
  \pinlabel* \fcolorbox{black}{white}{$pq|r$} at 77 72
  \pinlabel* \fcolorbox{black}{white}{$q|pr$} at 77 42
  \pinlabel* \fcolorbox{black}{white}{$qr|p$} at 50 26
  \pinlabel* \fcolorbox{black}{white}{$r|qp$} at 25 42
  \pinlabel* \fcolorbox{black}{white}{$rp|q$} at 25 72
  \pinlabel* \fcolorbox{black}{white}{$p|q|r$} at 70 91
  \pinlabel* \fcolorbox{black}{white}{$q|p|r$} at 90 57
  \pinlabel* \fcolorbox{black}{white}{$q|r|p$} at 70 23
  \pinlabel* \fcolorbox{black}{white}{$r|q|p$} at 30 23
  \pinlabel* \fcolorbox{black}{white}{$r|p|q$} at 10 57
  \pinlabel* \fcolorbox{black}{white}{$p|r|q$} at 30 91
 \endlabellist
 \centering
 \includegraphics[width=.45\textwidth]{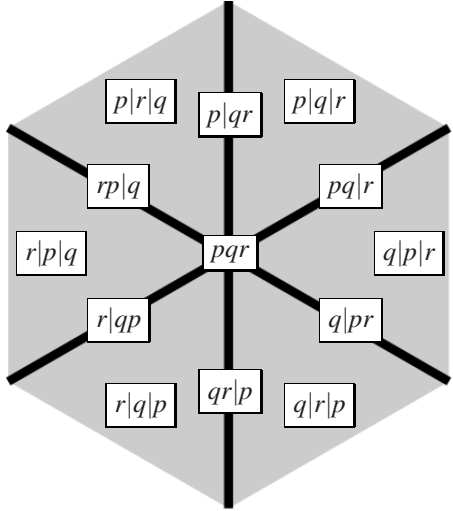}
 \caption{The permutahedron on $3$ letters, representing a neighborhood in parameter space of a triple critical value switch.} \label{F:3Perm} 
 \end{figure}

 Thus $F|_{\partial P_i}$ is divided into six elementary $1$--parameter families;
 we claim now that at most three of these can be of type $1$. If one believes this claim, then the handlebody filling associated to a point on $\partial P_i$, as one moves around $\partial P_i$, changes at most three times, and those changes (in the sense of which curves bound disks in the handlebody) all occur in a fixed genus $1$ subsurface $P$ of $\Sigma$ bounded by curves that themselves bound disks in the handlebody. Up to diffeomorphism, then, we either have no changes, or two changes, switching from a meridian bounding a disk to a longitude bounding a disk and then back, or three changes, switching from a meridian to a longitude to a $(1,\pm 1)$--curve on the torus. Zero changes gives $(X_g,\Sigma_g)$ and two changes gives $(X_{g-1},\Sigma_g)$, exactly as before. Three changes means that $\cP(F|_{\partial P_i})$ is described by a trisection diagram $(\Sigma,\alpha,\beta,\gamma)$ in which $\alpha$, $\beta$ and $\gamma$ only differ in one $T^2$ summand, which gives $(X^\pm_g, \Sigma^\pm_g)$ (see the trisection diagrams for $S^1 \times S^3$ and $\pm \C P^2$ in~\cite{GayKirbyTrisections}).

 It remains to prove that we have at most three type $1$ critical value switches around $\partial P_i$. In fact this claim follows from a simpler claim: Suppose that $f_t: P \to [-1,1]$ is a generic $1$--parameter family of functions on a surface $P$ of genus $1$ with $3$ boundary components, with $f_t^{-1}(1)$ equal to two boundary components of $P$ and $f_t^{-1}(-1)$ equal to one boundary component. Furthermore, suppose that $f_t$ has exactly three index $1$ critical points, and no other critical points, for all $t$, that there is exactly one critical value switch at $t=-1/2$ and one critical value switch at $t=1/2$, and $f_t$ is a generic Morse function for all other values of $t$, and finally that the two critical points that switch heights at $t=-1/2$ are not the same two critical points that switch at $t=1/2$. Then $f_t$ can not be of type $1$ on both intervals $[-1,0]$ and $[0,1]$. This is because, up to diffeomorphism, $f_0$ can be taken to be the vertical height function on one of the three genus $1$ surfaces embedded in $\R^3$ illustrated in Figure~\ref{F:3Genus1}. If, for example, $f_0(p) = 1/2$,  $f_0(q) = 0$, and $p$ and $q$ are to switch heights in a type $1$ switch, the component of $f_0^{-1}([-\e,1/2+\e])$ containing $p$ and $q$ needs to have genus $1$, but this can clearly not be true for both the pairs $(p,q)$ and $(q,r)$.
 \begin{figure}[h!]
 \labellist
  \hair 1.5pt
  \small
 \endlabellist
 \centering
 \includegraphics{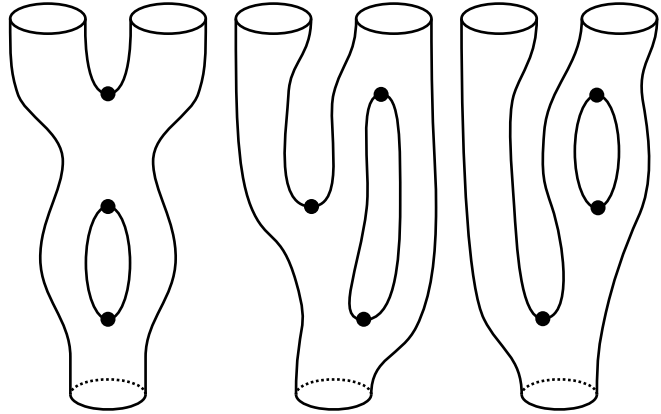}
 \caption{Three Morse functions on a genus $1$ surface with $3$ boundary components.} \label{F:3Genus1} 
 \end{figure} 

\end{proof}

\begin{definition} \label{D:PolyTypes01}
 An elementary polygon $P_i$, or an elementary $2$--parameter family $F = F|_{P_i}$ is of {\em type 0} if, in the above theorem, we have $\cP(F) \cong (X_g,\Sigma_g)$, is of {\em type 1} if $\cP(F) \cong (X_{g-1},\Sigma_g)$, and is of {\em type 2} if $\cP(F) \cong (X^\pm_g,\Sigma^\pm_g)$.
\end{definition}

\begin{proof}[Construction of $\cP_{B^2}$:]
 We parallel the construction of $\cP_{B^1}$ as closely as possible. Given a generic $F: B^2 \to C^\infty(\Sigma)$, subdivide $B^2$ into elementary polygons $B^2 = P_1 \cup \ldots \cup P_n$. Let $v_1, \ldots, v_l$ be the vertices and let $e_1, \ldots, e_m$ be the edges of this subdivision. Thicken the $0$-- and $1$--skeleton of this subdivision in the following sense: Let $V_i$ be a small disk neighborhood of each vertex $v_i$ (or a half-disk neighborhood when $v_i \in \partial B^2$). Shrink each $e_j$ so as to intersect the $V_i$ only at the endpoints of $e_j$ and let $E_j $ be a collar neighborhood of each $e_j$ (i.e. $E_j$ is identified with $I \times e_j$ where $I = [-\e,\e]$ or $[-\e,0] \times e_j$, according to whether $e_j$ is in the interior or boundary of $B^2$) which intersects the $V_i$ only at $I \times \partial e_j$. Finally shrink each $P_k$ so that $P_k$ intersects the union of the $V_i$ and the $E_j$ exactly along $\partial P_k$. 
 
 Let $Z_* = [-\epsilon,0] \times B^2 \times \Sigma$. For each vertex $v_i$, let $(H_i,\Sigma) = \cP(F|_{\{v_i\}})$, and let $Z_{v_i} = V_i \times H_i$. Let $Z_{**}$ be the result of attaching each $Z_{v_i}$ to $Z_*$ by identifying $\{t\} \times \Sigma \subset Z_{v_i}$ with $\{-\epsilon\} \times \{t\} \times \Sigma \subset Z_*$.
 
 Now, for each edge $e_j$, let $(X_j,\Sigma) = \cP(F|_{e_j})$ and let $Z_{e_j} = I \times X_j$ (where $I = [-\e,\e]$ or $[-\e,0]$ as above, according to whether $e_j$ is in the interior or boundary of $B^2$). Recall from the construction of $\cP_{B^1}$ that $X_j$ is built beginning with $[-\e,0] \times e_j \times \Sigma$, and thus for each $t \in e_j$ we have a particular copy of $\Sigma$ in $X_j$, namely $\{0\} \times \{t\} \times \Sigma$. Let $Z_{***}$ be the result of attaching each $Z_{e_j}$ to $Z_{**}$ as follows: For each $x \in I$ and each $t \in e_j$, first identify $\{x\} \times \{0\} \times \{t\} \times \Sigma$ in $X_j$ with $\{-\e\} \times \{(x,t)\} \times \Sigma$ in $Z_* \subset Z_{**}$. (The last $(x,t)$ is a point in $I \times e_j$ which is identified with $E_j \subset B^2$.) Also, recall that each $E_j$ intersects two vertex neighborhoods $V_i$ and $V_{i'}$ in intervals which can be naturally identified with $I \times \{v_i\}$ and $I \times \{v_{i'}\}$. Since $X_j$ contains at one end of its boundary a copy of the handlebody $H_i$ and at the other end a copy of the handlebody $H_{i'}$, we can next identify $I \times H_i$ in $Z_{e_j}$ with $I \times H_i$ in $Z_{v_i} \subset Z_{**}$ and also identify $I \times H_{i'}$ in $Z_{e_j}$ with $I \times H_{i'}$ in $Z_{v_{i'}} \subset Z_{**}$. This completes the construction of $Z_{***}$.

 Now note that $Z_{***}$ has one ``outer'' boundary component, diffeomorphic to $\cP(F|_{\partial B^2})$, and one ``inner'' boundary component for each elementary polygon $P_k$, diffeomorphic to $\cP(F|_{\partial P_k})$. Since we have classified the manifolds that arise as $\cP(F|_{\partial P_k})$, we know how to fill these extra boundary components in: In the type 0 case fill in with $\natural^g S^1 \times B^4$, in the type 1 case fill in with $\natural^{g-1} S^1 \times B^4$, and in the type 2 case fill in with $(\natural^{g-1} S^1 \times B^4) \natural ([0,1] \times (\pm \C P^2))$. In the type 2 case we are left with $\pm \C P^2$ boundary components, and in the other cases we completely cap off these boundaries.
 
 
\end{proof}

%
%
%
\bibliographystyle{plain}
%

%

\bibliography{Omega34}

\end{document}